\documentclass{cmslatex}
\usepackage[paperwidth=7in, paperheight=10in, margin=.875in]{geometry}
 \usepackage[backref,colorlinks,linkcolor=red,anchorcolor=green,citecolor=blue]{hyperref}

\usepackage{amsfonts,amsmath,amssymb}
\usepackage{graphicx}
\usepackage{epstopdf}
\usepackage{algorithmic}
\usepackage{paralist}
\usepackage{tikz}
\usepackage{ulem}
\newcounter{saveenum}
\newcommand{\vertiii}[1]{{\left\vert\kern-0.25ex\left\vert\kern-0.25ex\left\vert #1 
		\right\vert\kern-0.25ex\right\vert\kern-0.25ex\right\vert}}

\usepackage{macros}
\usepackage{cite}
\usepackage{enumerate}
\renewcommand{\theequation}{\arabic{section}.\arabic{equation}}

\begin{document}

\title{Friction mediated by transient elastic linkages : \newline
		asymptotic expansions and fat tails\thanks{July 3, 2024}}

\author{%
  Samar Allouch\thanks{Laboratoire Analyse, G\'eom\'etrie \& Applications, Universite Paris Sorbonne Nord, France},
  \and Vuk Milisic\thanks{CNRS, Univ Brest,  Laboratoire de Mathematiques de Bretagne Atlantique, UMR 6205, France}}

 \pagestyle{myheadings} \markboth{Friction mediated by transient linkages : expansions and fat tails}{S. Allouch \& V.Milisic} \maketitle

\begin{abstract}
Several results in  previous works , strongly depend on the
	exponential tail of the linkages' distribution in our adhesive models.
	 The purpose of this paper is to weaken this hypothesis and  to allow  
	 more {\em fat tails} for large ages. From the biological point of view
	 this means that we allow adhesions to be stronger, because linkages break less often.
	 
	 Moreover, in our previous articles, the asymptotic expansion of adhesion site's position 
	 and the corresponding error estimates also used some fast decay properties of the
	 kernel, we show, when the kernel is a given function of age but constant in time, how to overcome 
	 this problem and construct asymptotic expansions in a systematic way at any 
	 order with respect to a small parameter $\e$ representing the linkages' turnover.
\end{abstract}


\begin{keywords}
\textbf{integral equation, renewal problem, asymptotic expansion, non-exponential distribution, comparison principle, Volterra equation,
	singular perturbation problem}
\end{keywords}

\maketitle
\noindent{\bf{MSC calssification :} 35B40, 45D05}

\section{Introduction}
This article is part of a project related to the mathematical study of adhesion forces in the context of cell motility (see \cite{allouch2021friction}, \cite{MiOel.1}, \cite{MiOel.2}, \cite{MiOel.3} and \cite{MiOel.4}). Cell adhesion and migration play a crucial role in many biological phenomena such as embryonic development, inflammatory responses, wound healing and tumor metastasis. 

{The main motivation comes from the seminal papers \cite{OeSchVi,OeSch}, where the authors
	built a complex and realistic model of the Lamellipodium. It is a cytoskeletal quasi-two-dimensional actin mesh,  the whole structure propels the cell across a substrate.
	In these works, the authors represent the actin network through a set of 
	 2D axi-symmetric equations whose solution models Lamellipodium's 
	 position $x(s,t)$ evolving in time $t\in \rr$ and with respect to a reference configuration 
	 $s \in (0,1)$. 
Using gradient flow techniques, the authors obtain a system of non-linear equations :
$$
\left\{ 
\begin{aligned}
	&	\underbrace{\kappa^B \partial^2_s ( \eta \partial_s^2 x)}_{\text{bending}}-
	\underbrace{\partial_s(\eta \lambda \partial_s x)}_{\text{unextensibility constraint}} + \eta \boxed{ \underbrace{\mu^A D_t x }_{\text{adhesion }}} + \\
	&	\quad+	
	\underbrace{
		\partial_s
		\left( 
		\eta^2  \mu^T \left(
		\arccos
		\left(\partial_s \left| x \right|\right)-\varphi_0
		\right) \partial_s x^{\bot}
		\right)
	}_{\text{twisting}} 
	+ 
	\eta^2
	\boxed{ 
		\underbrace{
			\left(\mu^S D_t \varphi\right) x^\bot 
		}_{\text{stretching}}}
	=0 \\
	& 	\left| \partial_s x\right|= 1,\\
\end{aligned}
\right.
$$
supplemented by initial and boundary conditions that we omit for sake of conciseness.
The adhesion and stretching friction terms appearing in this force balance equation
were obtained as formal limits of a microscopic description of adhesion mechanisms.
For instance the first friction (boxed) term 
is obtained as the limit when $\e$, a small dimensionless parameter goes to zero in the expression :
$$
\ue \int_\rr \left(x_\e (s,t) - x_\e (s+\e a,t-\e a)\right) \kernel(s,a,t) da \to \mu_1 (s,t) \left( \dt - \ds \right) x_0(s,t) =: \mu^A D_t x
$$
here the $a$ represents the age of a linkage established with the previous locations
of the filament and the shift in the reference configuration $s+ \e a$ comes from 
the (de)polymerization of the filaments inside the lamelipodium. 
The  parameter $\e$ represents at the same time a characteristic age of the linkages, 
and  the inverse of the stiffeness of the elastic linkage under consideration.
The elongation $(x(s,t)-x(s+\e a,t-\e a))/\e$, for a fixed $a$, is the linear elastic force 
exerted on $x(s,t)$ by a linkage established at $x(s+\e a,t-\e a)$ (see also 
\cite{oelz_schmei_book,MR3385931} for an extension to the non-axi-symmetric case).

Our initial goal was to 
give a rigorous justification of these formal computations but the non-linear
space-dependent feature of the problem made it out of reach at that time.
Instead we considered a simplified problem where only a single adhesion 
point $\xeps(t)$ was taken into account \cite{MiOel.1}.
}
It is obtained  solving a simplified force balance  equation 
\begin{equation}\label{eq.X.eps}
	\left\{
	\begin{aligned}
		&\ue \int_\rr \left(X_\e (t) - X_\e(t-\e a)\right)\rhoe(a,t) da = f(t), & t>0, \\
		&X_\e(t) = X_p(t),& t\leq 0,
	\end{aligned}
	\right.
\end{equation}
where $f$ is an external force and the left hand side {is} a 
continuum of elastic spring forces with respect to past positions {$\xeps(t-\e a)$}.
The density of linkages 
$ \rhoe$ is either a given function or it can be, {for instance as in the original model \cite{OeSch},}  a solution of an age-structured problem~:
\begin{equation}\label{eq.rho.complet.bis}
	\left\{
	\begin{aligned}
		&\left(\e \dt + \da + \zeta(a,t)\right) 	 \rhoe(a,t) = 0 , & (a,t) \in \rr\times(0,T), \\
		&\rhoe(0,t) = \beta(t) \left( 1 - \int_\rr \rhoe(a,t) da\right), & (a,t) \in\{0\}\times(0,T), \\
		&\rhoe(a,0) = \rhoi(a), & (a,t) \in \rr \times\{0\}.
	\end{aligned}\right.
\end{equation}
In the latter system, $\beta$ (resp.  $ \zeta \in \mathbb{R}_{+} $) is the kinetic on-rate (resp. off-rate) (\cite{MiOel.1}, \cite{MiOel.2}, and \cite{MiOel.3}). Under the {major} assumption that the off-rate $\zeta$ admits a strictly positive lower bound 
{$$
0< \ztmin  \leq \zeta(a,t), \quad \alev (a,t) \in \rr\times\rr
$$}
 in \cite{MiOel.1}, the authors  studied  the asymptotic limit of the systems \eqref{eq.X.eps} and \eqref{eq.rho.complet.bis} when $\e$ goes to zero. They obtained  {rigorous} convergence results{, namely they showed that}
$$ \| X_{\e}-X_{0}\|_{{C}^{0}([0,T])}+\| \rho_{\e}-\rho_{0}\|_{{C}^{0}((0,T]; L^{1}(\mathbb{R}_{+}))} \to 0,$$
where the limit 	$ X_{0}$ solves {a simple force balance between the external load and a
friction term, }
\begin{equation}\label{eqX0}
\left\{
	\begin{aligned}
		& \mu_{1,0}(t)X'_{0}=f(t),&  t>0,\\ 
		& X_{0}(t=0)=X_{p}(0),& t=0,
	\end{aligned}
\right.
\end{equation}
here $\mu_{1,0}(t)$ denotes the first moment of $\rhoz$, {\em i.e.} $\mu_{1,0}(t):= \int_{0}^{\infty} a \rho_{0}(a,t)da$ and  $\rho_{0}$ satisfies {the non-local problem}~:
\begin{equation}\label{eq.rho.zero}
	\left\{
	\begin{aligned}
		&{{\partial}_a}\rho_{0}+\zeta(a,t)\rho_{0}=0 ,\hspace{4cm} &t >0, a >0, \\ 
		&\rho_{0}(a=0,t)=\beta(t)\left( 1- \int_{0}^{\infty}\rho_{0}(t,\Tilde{a})~d\Tilde{a} \right),& t >0.
	\end{aligned}\right.
\end{equation}
We underline that, {if for instance $f$ is constant or admits some limit for $t$ large}, 
the $\e$ scaling {provides some convergence results for} long time asymptotics 
of \eqref{eq.X.eps} and \eqref{eq.rho.complet.bis} {(see for more details \cite[Theorem 4.4]{milisic-shmeiser})}.

{In the previous setting, the convergence of $\rhoe$ towards $\rhoz$ stemmed form \cite[Lemma 3.1 and Theorem 3.1, p. 495]{MiOel.1}.  Using the  very particular Lyapunov functional
	\begin{equation}\label{lyapunov-fct}
		\mathcal{H}[u](t):=\left| \int_{0}^{\infty}u(a,t)da\right| + \int_{0}^{\infty} | u(a,t)|da~,
	\end{equation}
the authors obtained, setting $\hrho(a,t) := \rhoe(a,t)-\rhoz(a,t)$ that
\begin{equation}\label{eq.nrj}
	\e \ddt{} \cH [\hrho] (t ) + \ztmin \cH[\hrho] (t ) \leq o_\e (1)
\end{equation}
leading to exponential error estimates :
$$
\| \rhoe(\cdot,t) - \rhoz(\cdot,t) \|_{L^1(\rr)} \leq \exp( - t \ztmin/\e) \| \rhoi(\cdot) - \rhoz(\cdot,0) \|_{L^1(\rr)} + o_\e(1)
$$
where $o_\e(1)$ means small when $\e$ is small. The exponential nature of the initial layer is closely
related to the strictly positive lower bound $\ztmin$. This convergence result implies that, 
for $\e$ small enough,  $\rhoe$ is a perturbation
of an exponentially decreasing profile in age. 
From the modelling point of view, the exponential {\em tail} of 
$\rhoe$  is of importance since it restricts the linkages to be of very short range 
 in \eqref{eq.X.eps}.
On the  mathematical side, if instead $\zeta$ is non-negative but does not admit 
a lower bound this convergence result is to our knowledge an open question. Indeed, 
the Lyapunov functional then provides only stability; no convergence result is at hand since
no lower-bound on the dissipation exists in \eqref{eq.nrj}.
This work intends to 
fill this gap. Namely assuming that there exists a non-increasing positive function $m \in L^1(\rr)$, such that 
$$
0\leq - \frac{m'(a)}{m(a)} \leq \zeta(t,a) \leq \ztmax ,\quad \text{ a.e.} \; (a,t) \in \rr\times\rr
$$
we show that it is again possible to establish convergence of $\rhoe$ towards $\rhoz$.
For instance let's examin some particular choices of $m$.
Choosing $m(a) := \exp(-\ztmin a)$, one recovers the exponential case
previously studied, whereas if $m(a) := (1+a)^{-\sigma}$ for $\sigma > 0$, allows 
$\zeta$ to tend to zero as $a$ grows.

This new convergence result is made possible 
thanks to a higher order asymptotic expansion of $\rhoe$ which includes
a first order {\em far field} term $\rhou$ 
solution of :
$$
	\left\{
	\begin{aligned}
		& \left(\da + \zeta(a,t)\right) \rhou(a,t) = - \dt \rhoz(a,t), & a>0,\; t>0,\\
		& \rhou(0,t) = -\beta(t) \int_\rr \rhou(a,t) da,& a=0,\; t>0, 
	\end{aligned}
	\right.
$$
and the initial layer (of order zero)  $\rz$, solving :
$$
	\left\{
	\begin{aligned}
		& \left(\dt + \da + \zeta(a,0) \right) \mrho_0(a,t) = 0, & a >0, \; t>0, \\
		& \mrho_0(0,t) = - \beta(0) \int_\rr \mrho_0(a,t) da,  & a=0,\; t>0,\\
		& \mrho_0(a,0) = \rhoi(a) - \rhoz(a,0),& a>0,\; t=0.
	\end{aligned}
	\right.
$$
We study 
finely the decay properties of this initial layer. First, we characterize in a new way the 
behavior of the trace $\mrho_0(0,t)$ in some weighted $L^1(\rr;(1+t)^\alpha)$-space.
As the trace solves a convolution Volterra equation of the second kind, the concept of resolvent appears naturally \cite[Chapters 2, 4, 7, 9]{Gripenberg_ea}
and its specific properties are used there.
These results combined with Duhamel's formula allow to show some stability
of $\mrho_0$ in other weighted $L^1$-spaces in age and in time (see Section \ref{sec.initial.layer} for more precise statements).
By increasing the order of the asymptotic expansion, we can use the {\em a priori} estimates above, without dissipation, and still prove a convergence theorem.
Indeed, for any $t>0$, one has
$$ \cH[\rhoe(\cdot,t)-\rhoz(\cdot,t)-\e \rhou(\cdot,t)-\rz(\cdot,t/\e)]\lesssim  o_{\e}(1), $$ 
and the microscopic features of the initial layer (using again Section \ref{sec.initial.layer}) allow then to recover~:
$$\|\rhoe-\rhoz\|_{L^1(\rr\times(0,T),(1+a)))}\leq o_\e(1). $$ 

In \cite[Theorem 1.1]{MiOel.1}, the convergence of $\Xeps$ towards $\Xz$ is also established through error estimates 
thanks to a comparison principle specific to non-negative Volterra non-convolution kernels \cite[Chap. 9]{Gripenberg_ea}. In the proof of this theorem, again a restrictive hypothesis appears
on $\zeta$ \cite[Assumption 1.1 (i) p. 487]{MiOel.1} as well. Indeed, at some
point in the proof of \cite[Theorem 1.1]{MiOel.1}, one needs to control the magnitude of
$$
A_\e \left[\rhoe\right](t):=\frac{\int_\rr a \kernel(a+t/\e,t) da}{\int_\rr \kernel(a+t/\e,t) da}
$$
cf \cite[Lemma 2.4, p. 491-492]{MiOel.1}, this is possible if $\zeta(a+t/\e,t)$ is increasing
for $a>a_0$ where $a_0>0$ is given, for a.e. fixed $t>0$. 
Using weak-* convergence arguments from \cite{MiOel.2} on the elongation problem,  
we prove convergence of $\Xeps$ without any restriction related to $\zeta$.

Nevertheless, if we still  need  to obtain 
error estimates on $\Xeps$, in the case where $\kernel$ is a  given non-negative kernel
{\em constant in time}, we show, in a systematic way, how to construct 
an asymptotic expansion with respect to $\e$ of $\Xeps$ and obtain error estimates.
The novelty here is that we only use moments of $\kernel$, namely that for a fixed integer $N\geq 1$, 
$$
\int_\rr (1+a)^{N+2} \kernel(a) da < + \infty.
$$
We construct  a  $N^{th}$-order asymptotic approximation
\begin{equation}\label{approx_dev_Xe}
	\ov{X}_{\e,N}=X_{O}(t)+ X_{I}(\tau)+O(\e^N),
\end{equation}
where $X_O$ (resp. $X_I$) represent a far field 
(resp. near field) Ansatz and $\tau=t/\e$ is the stretched variable. Then we prove that
$$
\nrm{X_\e- \ov{X}_{\e,N}}{C([0,T])} \lesssim \e^N.
$$
The key point is that we avoid the need of evaluating ratios such as $A_\e[\rhoe](t)$ 
by a refined analysis of signed initial layers (see the proof of Theorem \ref{thm.main.first.part} and Remark \ref{rmk.super.sol} below).
}

The paper is structured as follows :  {in section \ref{section5}, we analyze the asymptotic behavior of $(\rhoe,X_\e)$ solutions of the weakly coupled problem \eqref{eq.rho.complet.bis}-\eqref{eq.X.eps} when $\e$ tends to $0$. 
In section \ref{section.3}, we analyze the  case where the kernel in \eqref{eq.X.eps}  depends on the age variable only but does not solve any specific equation. 
Moreover, we construct the asymptotic expansion of $X_\e$ and show, under regularity hypotheses on the data,  error estimates up to any order.}

\section{{The weakly coupled case}}\label{section5}

{We refer to \cite{Zie.Book} for 
a general framework for $\rm BV$ functions, our notations being coherent 
with this reference.}

\begin{assum}\label{hypo.data.general}
	Assume the kernel $\rhoe$ solves \eqref{eq.rho.complet.bis} and that the data satisfies :
	\begin{enumerate}[a)]
		\item the off-rate $\zeta$ is in $C([0,T];L^{\infty}(\rr))$ and there exists a non-increasing $m\in L^{1}(\rr;(1+a)^3)$ such that
		$$
		- \frac{m'(a)}{m(a)} \leq \zeta(a,t) \leq \ztmax ,\quad \alev a \in \rr.
		$$
		\item the birth-rate $\beta \in C(\rr)$ is such that
		$$
		0 < \bmin \leq \beta(t) \leq \bmax.
		$$
		\item the initial condition $\rho_{I}\in BV(\rr)\cap {L}^{\infty}(\rr)\cap L^{1}(\rr,(1+a)^3) $ 
		{is such that there exists $C_{\rhoi}>0$ so that} 
		$$\rho_{I}(a)\leq C_{\rhoi} \,  m(a), ~\text{for a.e. }~a\in \rr. $$
		\setcounter{saveenum}{\value{enumi}}
	\end{enumerate}
	Whereas $\Xeps$ solves \eqref{eq.X.eps} and the data satisfy
	\begin{enumerate}[a)]
		\setcounter{enumi}{\value{saveenum}}
		\item the source term  $f$ belongs to $C^2(\rr)$,
		\item the past data $\Xp \in L^\infty (\RR_-) \cap \Lip(\RR_-)$.
	\end{enumerate}
\end{assum}

\subsection{{Asymptotic expansion of linkages}}
We introduce {a} first order asymptotic approximation of $\rhoe$ solving \eqref{eq.rho.complet.bis}~:
\begin{equation}\label{rho.tilde.eps}
	\trhoe(a,t):=\rhoz(a,t) - \trho_0 (a,t) - \e \rhou(a,t)\end{equation}
where $\rhoz$ is the zeroth order macroscopic limit solving \eqref{eq.rho.zero},
$\rhou $ is the first order macroscopic profile given when solving 
\begin{equation}\label{eq.rho.un}
	\left\{
	\begin{aligned}
		& \left(\da + \zeta(a,t)\right) \rhou(a,t) = - \dt \rhoz(a,t), & a>0,\; t>0,\\
		& \rhou(0,t) = -\beta(t) \int_\rr \rhou(a,t) da,& a=0,\; t>0, 
	\end{aligned}
	\right.
\end{equation}
and $\trho_0(a,t) := \rz(a,t/\e)${ is  the {(rescaled)} initial layer $\rz$ solving
	\begin{equation}\label{eq.initial.layer}
		\left\{
		\begin{aligned}
			& \left(\dt + \da + \zeta(a,0) \right) \mrho_0(a,t) = 0, & a >0, \; t>0, \\
			& \mrho_0(0,t) = - \beta(0) \int_\rr \mrho_0(a,t) da,  & a=0,\; t>0,\\
			& \mrho_0(a,0) = \rhoi(a) - \rhoz(a,0),& a>0,\; t=0.
		\end{aligned}
		\right.
	\end{equation}  
	
}

\subsubsection{{Decay properties of outer expansions}}

\begin{propm}\label{prop.moments}
	Let Assumptions \ref{hypo.data.general} hold, then there exists generic constants $c_j>0$ and $c_{j,1}>0$, such that $\rhoz$ (resp. $\rhou$) solution of \eqref{eq.rho.zero}  (resp. \eqref{eq.rho.un}) satisfies 
	$$
	\left|\rho_j (a,t) \right|\leq c_j (1+a)^{2j} m(a),~\text{for}~j\in\{0,1\}.
	$$
	In a generic way, for all $j\in\{0,1\}~\text{and}~k\in \NN $, if $ \zeta\in W^{k,\infty}(\rr\times \rr)$ then~:
	$$
	\left| \partial_{t^k}^k \rho_{j} (a,t)\right| \leq c_{j,k} (1+a)^{2j+k} m(a).
	$$
\end{propm}
The proof uses equations \eqref{eq.rho.zero} and \eqref{eq.rho.un} together with  Assumptions \ref{hypo.data.general}
and is postponed to Appendix \ref{sec.proof}.
\subsubsection{{A complete study of the initial layer}}\label{sec.initial.layer}

One considers the problem \eqref{eq.initial.layer}
and defines $x(t):=\mathfrak{r}_0(0,t)$, then by using Duhamel's 
principle {\eqref{eq.initial.layer}} can be rewritten as 
\begin{equation}\label{eq.x.rho}
	\left\{
	\begin{aligned}
		& x+k\star x=b,\\
		& k(a):=\beta{(0)}\exp\left( -\int_{0}^{a}\zeta(\tau,0)d\tau\right),\\
		& b(t):=-\beta{(0)} \int_{t}^{\infty}\rz(a-t,0)\exp\left( -\int_{a}^{a-t}\zeta(\tau,0)d\tau\right)da.
	\end{aligned}\right.
\end{equation}
As a consequence of the Paley-Wiener theorem \cite[Theorem  4.1]{Gripenberg_ea} and the fact that $k$ is a decreasing function of $a$, \cite[p.264]{Gripenberg_ea}~:
\begin{thmm}
	If $k$ is a decreasing non-negative kernel such that $k\in L^{1}(\rr)$, then the resolvent associated to \eqref{eq.initial.layer} satisfies~: 
	$r+r\star k=k$ and $r\in L^{1}(\rr)$.	{Moreover, the solution is unique, explicit and reads :
	$$
	x=b-r\star b
$$
}
\end{thmm}

\begin{propm}\label{prop.trho}
	Let Assumptions \ref{hypo.data.general} hold. If moreover, $\rz(\cdot,0) \in L^1(\rr,(1+a)^2)$ 
	and  there exists a constant ${C_\rz}>0$ such that
	$$
	\rz(a,0) \leq {C_\rz} m(a),
	$$
	then
	{\begin{compactenum}[i)]
		\item  $x = \rz(0,\cdot) \in L^1_t(\rr;(1+t)^2)\cap L^\infty (\rr)$,  
		\item $
		\rz \in L^1(\rr \times \rr;(1+a)),
		$
		\item  there exists another  constant $c'>0$ such that
		$$
		\rz (a,t) \leq c' m(a).
		$$
	\end{compactenum}}
\end{propm}

\begin{proof}
	{By standard fixed point arguments \cite[Theorem 2.1, p. 497]{MiOel.1}, there exists a unique solution $\rz \in C(\rr;L^1(\rr))$, 
	Using \cite[Theorem 2.3.5 p. 44]{Gripenberg_ea}, as $b \in C(\rr)$, so is $x$.}
	Using the Assumption \ref{hypo.data.general}.a) on $\zeta$ and on the data, 
	{the source term $b$ in \eqref{eq.x.rho} satisfies~:}
	$$
	\left|b(t)\right| \leq \beta \int_t^\infty \rz(a-t,0) \frac{m(a)}{m(a-t)} da \leq \beta C_{\rz} \int_t^\infty m(a) da  
	$$
	which is bounded since $m\in L^1(\rr)$ and it is also integrable because so is the first moment of $m$. Writing then that
	$$
	x= b - r\star b
	$$
	and since $L^1$ is an algebra for the convolution, $x \in L^1(\rr)$. Because, $b$ is bounded and $r\in L^1{(\rr)}$, $x \in L^\infty(\rr)$ as well. {This proves partially ii), the part concerning higher moment with respect to time is postponed to the end of the proof.}

	Then in a similar way, using Duhamel's principle, one has~:
	$$
	\rz(a,t) := \begin{cases}
		\rz(0,t-a) \exp\left( - \int_0^a \zeta(\tia,0) d\tia \right), & \text{ if } t>a, \\
		\rz(a-t,0)  \exp\left( - \int_{a-t}^a \zeta(\tia,0) d\tia \right), & \text{ otherwise },
	\end{cases}
	$$
	so that
	$$
	\left|\rz(a,t)\right| \leq  \begin{cases}
		\nrm{x}{L^\infty} \frac{m(a)}{m(0)}, & \text{ if } t>a, \\
		\frac{\left|\rz(a-t,0)\right|}{m(a-t)} m(a) \leq {C_{\rz}} m(a), & \text{ otherwise },
	\end{cases}
	$$
	which gives $c'$ in the last estimates of the claim {iii)}. Next we consider~:
	$$\begin{aligned}
		\int_\rr \left|\rz(a,t) \right|& da \leq 
		\int_0^t \left|\rz(0,t-a) \right|\exp\left( - \int_0^a \zeta(\tia,0) d\tia \right) da \\
		& + \int_t^\infty \left|\rz(a-t,0)\right|  \exp\left( - \int_{a-t}^a \zeta(\tia,0) d\tia \right) da \\
		&\leq  \int_0^t\left| \rz(0,t-a) \right|\frac{m(a)}{m(0)} da + \int_t^\infty \rz(a-t,0) \frac{m(a)}{m(a-t)} da \\
		& = \frac{1}{m(0)} \left\{\left(\left|\rz(0,\cdot) \right|\star m\right) (t) + \int_\rr \frac{\left|\rz(\tia,0)\right|}{m(\tia)} m(\tia+t) d\tia \right\} \\
		& = \frac{1}{m(0)} \left(\left|x\right| \star m\right) (t) +  {C_{\rz}} \int_t^\infty m(a) da
	\end{aligned}
	$$
	Then using that $L^1(\rr)$ is an algebra for the convolution, and that $m\in L^1(\rr,(1+a))$, one concludes that
	$$
	\int_\rr \int_\rr \rz(a,t) da dt < \frac{\nrm{m}{L^1(\rr,(1+a))} }{m(0)} \left(\nrm{\rz(0,\cdot)}{L^1(\rr)} +{C_{\rz}} m(0)
	\right)
	$$
	the same holds for the first moment as well. {This proves ii).}

	Setting $q_1(a,t) := t \rz (a,t)$,
	{it solves
	$$
	\left\{
	\begin{aligned}
		& 	\left(\dt + \da +\zeta(a,0)\right)q_1(a,t) = \rz (a,t)	\\
		&	q_1(0,t) = - \beta(0) \int_\rr q_1 (\tia,t)	d\tia \\
		&	q_1(a,0)	 =0	 
	\end{aligned}
	\right.
	$$
in the sense of characteristics \cite[Lemma 2.1, p. 489]{MiOel.1}}.
Defining the trace $y(t) := q_1(0,t)$,{ and using Duhamel's principle, it can be shown
that $y$ solves : }
	$$
	\begin{aligned}
		y(t) = - & \beta(0) \int_0^t y(t-a) \exp\left(-\int_0^a \zeta(s,0) ds\right) da \\
		& - \beta(0) \int_0^t a x(t-a) \exp\left(-\int_0^a \zeta(s,0) ds\right) da \\
		&- \beta \int_t^\infty t \rz(a-t,0) \exp\left(-\int_{a-t}^{a} \zeta(s) ds\right) da 
		& {=: - (k\star y)(t)+ b_x(t)  }
	\end{aligned}
	$$
	{\em i.e.}
	$
	y+k\star y = b_x
	$.
	Assuming that $\int_\rr (1+a)^2 m(a) da<\infty$ shows that  $b_x \in L^1(\rr)$. 
	Then as above, 
	$$
	\begin{aligned}
		\int_0^t& \left| q_1(a,t)\right| da  \leq \int_0^t \left(\left|q_1(0,t-a) \right|+ a \left|x(t-a)\right| \right)\exp\left(-\int_0^a \zeta(s,0)ds\right) da \\
		& \leq C \left|q_1(0,\cdot)\right| \star m + \left|x(0,\cdot)\right| \star (\cdot) \,m (\cdot)
	\end{aligned} 
	$$
	together with
	$$
	\int_t^\infty \left|q_1(a,t)\right| da \leq t \int_t^\infty \mrho_0(a-t,0)\exp\left( - \int_{a-t}^a \zeta(\tau,0)d\tau\right) da
	\leq t {C_{\rz}} \int_t^\infty m(a) da
	$$
	both left-hand sides are then $L^1_t(\rr)$ functions in time,  provided that $m \in L^1_a(\rr,(1+a)^2)$. For the next step, one works similarly and obtains that
	$$
	\int_\rr \int_\rr t^2 \rz(a,t) da dt < \infty
	$$
	since $m$ is in $L^1_a(\rr,(1+a)^3)$. {This finishes the proof of i)}.
\end{proof}

\begin{corom}\label{corom.moment.initial.layer}
	Under the Assumptions \ref{hypo.data.general}, one has  $x(t):=\rz(0,t)\to 0$ when $t$ grows large and the first moment can be estimated as follows~:
	$$ \int_{\rr}(1+a)\left| \rz(a,t)\right|da\leq o(1)+c\int_{t}^{\infty}(1+a)m(a)da,$$
	where $o(1)$ denotes small when $t$ grows large.
\end{corom}
\begin{proof}
	Using Lyapunov's functional \eqref{lyapunov-fct}, one has that
	$$\int_{\rr}\left| \rz(a,t)\right|da\leq \infty$$
	which, thanks to the boundary condition, provides that $x(t)$ is bounded on $\rr$. This shows using Duhamel's principle that $\rz\in L^{\infty}(\rr\times\rr)$.
	Moreover, defining the discrete differences 
	$$ 
	\Drz(a,t):=\cfrac{\rz(a,t+h)-\rz(a,t)}{h}
	$$
	 solve the problem~:
	\begin{equation}\label{eq.disc.diff}
		\left\{
		\begin{aligned}
			& \left(\dt + \da + \zeta(a,0) \right) \Drz = 0, & a >0, \; t>0, \\
			& \Drz(0,t) = - \beta(0) \int_\rr \Drz(a,t) da,  & a=0,\; t>0,\\
			& \Drz(a,0) = \Frac{\rz(a,h)-\rz(a,0)}{h},& a>0,\; t=0.
		\end{aligned}
		\right.
	\end{equation}
	Thanks to Assumptions \ref{hypo.data.general} {and Proposition \ref{prop.moments}}, one shows that 
	$\rz(\cdot,0):=$ ${(\rhoi(\cdot)}$ ${-\rhoz(\cdot,0))}$ $\in BV(\rr)$. Moreover, one has that 
	$$
	\left|D^{h}_{t}x(t) \right|=\left|\Drz(0,t) \right|\leq \beta(0)\int_{\rr}\left|\Drz(a,t) \right|da\leq \beta(0)\mathcal{H}[\Drz(a,t)] 
	$$
	{similarly as in \cite[Lemma 3.1 p. 493, and Lemma 3.3, p. 495]{MiOel.1}}, 
	we obtain that, for all $t>0$, 
	$$ 
	\mathcal{H}[\Drz(\cdot,t)] \leq \mathcal{H}[\Drz(\cdot,0)].
	$$
	It suffices then to prove that the initial term $\mathcal{H}[\Drz(\cdot,0)] $ is bounded. Indeed, we have 
	$$ \mathcal{H}[\Drz(a,0)]=\int_{\rr}\left| \Frac{\rz(a,h)-\rz(a,0)}{h}\right|da+\left|\Frac{\mu_{\rz}(h)-\mu_{\rz}(0)}{h}\right|:=I_1+I_2$$
	where $\mu_{\rz}(t):=\int_{\rr}\rz(a,t)da $. For the first term, we split the integral in two parts
	\begin{equation*}
		\begin{aligned}
			I_1&=\int_{0}^{h}\left| \Frac{\rz(a,h)-\rz(a,0)}{h}\right|da  +\int_{h}^{+\infty}\left| \Frac{\rz(a,h)-\rz(a,0)}{h}\right|da\\
			&= \Frac{1}{h}\int_{0}^{h}\left| \rz(0,h-a)\exp\left( -\int_{0}^{a}\zeta(\tia,0)d\tia\right)-\rz(a,0)\right|da \\
			&+ \Frac{1}{h}\int_{h}^{+\infty}\left| \rz(a-h,0)\exp\left( -\int_{a-h}^{a}\zeta(\tia,0)d\tia\right)-\rz(a,0)\right|da=:I_{1,1}+I_{1,2}
		\end{aligned}
	\end{equation*}
	where we used Duhamel's principle. It easy to see that the first term 
	$$ I_{1,1}\lesssim\Frac{1}{h}\int_{0}^{h}da\left(\left| \beta(0)\right|\|\rz\|_{L^{\infty}_{t}L^{1}_{a}}+ \|\rz(a,0)\|_{L^{\infty}(\rr)} \right)< \infty.$$
	Concerning $I_{1,2}$, one splits the integral adding and subtracting intermediate terms
	\begin{equation*}
		\begin{aligned}
			I_{1,2}&\leq \Frac{1}{h}\int_{h}^{+\infty}\left| \left(\rz(a-h,0)
			-\rz(a,0) \right)\exp\left( -\int_{a-h}^{a}\zeta(\tia,0)d\tia\right)\right|da\\
			&+ \Frac{1}{h}\int_{h}^{+\infty}\left| \rz(a,0)\exp\left( -\int_{a-h}^{a}\zeta(\tia,0)d\tia\right)-\rz(a,0)\right|da\\
			&\lesssim  TV(\rz(\cdot,0))+{\ztmax}\|\rz(a,0)\|_{L^{1}(\rr)}
		\end{aligned}
	\end{equation*}
	where $TV$ denotes total variation of $\rz(\cdot,0)$ \cite{Jonel}. For the second term $I_2$ noting that $$\left|\dt\mu_{\rz}(t)\right|\leq \left( \zeta_{\max}+\beta_{\max}\right)\|\rz\|_{L^{\infty}_{t}L^{1}_{a}},$$ one obtains 
	$$ I_2\leq \left( \zeta_{\max}+\beta_{\max}\right)\|\rz\|_{L^{\infty}_{t}L^{1}_{a}},$$
	and finally, we obtain that $ \mathcal{H}[\Drz(a,t)] < \infty$, for all $t>0$, which shows that $x\in \Lip(\rr)$. Since $x\in L^{1}(\rr)$ (see Proposition \ref{prop.trho}), this implies that $\lim_{t\to+\infty}x(t)=0$. Now, we consider 
	\begin{equation*}
		\begin{aligned}
			J(t):=\int_{0}^t(1+a)& \left| \rz(a,t)\right|da\leq \int_{0}^{t} \left|x(t-a) \right|(1+a)m(a)da\\
			&=\left( \int_{0}^{t/2}+\int_{t/2}^{t}  \right)\left|x(t-a) \right|(1+a)m(a)da=:J_1+J_2.
		\end{aligned}
	\end{equation*}
	For every $\delta>0$ there exists $\eta_{1}$ such that $t>\eta_1$, implying that 
	$$ \sup_{s\in (\frac{t}{2},t)}\left| x(s)\right|<\Frac{\delta}{\left(2\int_{\rr}(1+a)m(a)da\right)} $$
	which shows that $J_1<\delta/2$. On the other hand, there exists $\eta_2$ such that $t>\eta_2$ which implies that
	$$ J_2\leq \|x\|_{L^{\infty}(\rr)}\int_{t/2}^{t}(1+a)m(a)da< \delta/2, $$
	by Lebesgue's Theorem (since the integral of $(1+a)m(a)$ is finite). These arguments show that $J(t)$ vanishes when $t$ grows large. On the other hand, {from Proposition \ref{prop.trho} iii)}~:
	$$ \int_{t}^{\infty}(1+a)\left| \rz(a,t)\right|da\leq c' \int_{t}^{\infty}(1+a)m(a)da,$$ which is an initial layer.
\end{proof}
\subsubsection{Error estimates for the linkage's density}
We define the difference 
$
\hrhoe(a,t) := \rhoe(a,t) -  \trhoe(a,t)
$ where $\trhoe $ is defined by \eqref{rho.tilde.eps}.
We obtain~:
\begin{thmm}\label{thm.error.estimates}
	Under Assumptions \ref{hypo.data.general}, 
	one has that
	\begin{equation}\label{eq.est.err.H}
		\cH[\hrhoe] (t) \leq  o_\e(1) ,\quad \alev t \in \rr
	\end{equation}
	where $\cH[u] (t) := \int_\rr \left| u(a)\right| da + \left| \int_\rr u(a) da\right|$
	{and  $o_\e(1)$ means small when $\e$ is small.}.
\end{thmm}

\begin{proof}
	We write the system satisfied by $\hrhoe$~:
	\begin{equation}\label{eq.error.zero}
		\left\{
		\begin{aligned}
			&\begin{aligned}
				(\e \dt + \da + \zeta(a,t)) \hrhoe(a,t) = 
				- \e  \left( \zeta(a,t) - \zeta(a,0)\right) \trho_0(a,t) - \e^2 \dt \rhou 
			\end{aligned} \\
			& \begin{aligned}
				\hrhoe(0,t) =&  - \beta(t) \int_\rr \hrhoe(a,t) da 
				- \left(\beta(t) - \beta(0)  \right) \int_\rr \trho_0(a,t) da 
			\end{aligned}\\
			& \hrhoe(a,0) = - \e \rhou(a,0)
		\end{aligned}
		\right.
	\end{equation}
	Then following the same steps as in {\cite[Lemma 3.1 and 3.3, p. 493-495]{MiOel.1}}, one has that
	$$
	\begin{aligned}
		\e \ddt{} &\cH[\hrhoe](t) + \int_\rr \zeta(a,t)\left\{ \left| \hrhoe(a,t)\right| +\hrhoe(a,t) \sgn\left( \int_\rr \hrhoe(\tia,t)d\tia\right)\right\} da \\
		& \leq 2 \e^2\int_{\rr} \left| \dt\rhou(a,t)\right|da+2\int_{\rr}\left|\zeta(a,t) - \zeta(a,0) \right|\left|\trho_0(a,t) \right|da \\
		&+\left| \beta(t)-\beta(0)\right|\int_{\rr}\left| \trho_0(a,t)\right|da
	\end{aligned}
	$$
	which after integration in time provides~:
	$$
	\begin{aligned}
&		\cH[\hrhoe](t) \leq  \cH[\hrhoe](0)  + 2 \e \int_0^t \int_\rr \left| \dt \rhou(a,s)\right| da ds \\
		& +  2 \int_0^{\tse} \int_{\rr}\left|\zeta(a,\e \tit) - \zeta(a,0) \right|\left|\rz(a,\tit) \right|da d\tit 
		 +2 \int_0^{\tse} \left| \beta(\e \tit)-\beta(0)\right|\int_{\rr}\left| \rz(a,\tit)\right|da d\tit.
	\end{aligned}
	$$
	Now, here the crucial point is that, thanks to Lebesgue's Theorem, the last two terms of the right-hand side do tend to zero as $\e$ goes to zero.
\end{proof}

{Thanks to the study of initial layer's properties from Section \ref{sec.initial.layer}, one can
	recover a sharper result on the zeroth order expansion of $\rhoe$ estimating the contribution of higher order terms.}
\begin{corom}\label{cor.estimate}
	Let Assumptions \ref{hypo.data.general} hold, then one has
	$$ \int_{\rr}(1+a)\left|\rhoe(a,t)-\rhoz(a,t) \right|da\leq o_{\e}(1)+\int_{\tse}^{\infty} (1+a)m(a)da.$$
\end{corom}
\begin{proof}
	Considering the system solved by the difference $e(a,t):=\rhoe(a,t)-\rhoz(a,t)$
	
	\begin{equation}\label{eq.error.e}
		\left\{
		\begin{aligned}
&				(\e \dt + \da + \zeta(a,t)) e(a,t) = \e \dt \rhoz(a,t), &a>0, \; t>0, \\
&				e(0,t) =  - \beta(t) \int_\rr \left(\rhoe(\tia,t)-\rhoz(\tia,t)-\trho_0(\tia,t)\right) da \\  & \qquad\qquad
				-\beta(t) \int_\rr\trho_0(a,t) da, & a=0, \;t>0, \\
			& e(a,0) = \rhoi(a)-\rhoz(a,0), &a>0,  \;t=0.
		\end{aligned}
		\right.
	\end{equation}
	It satisfies \eqref{eq.error.e} in the sense of characteristics \cite[Lemma 2.1 and 3.1]{MiOel.1}, namely
	
	\begin{equation}\label{duhamel.formula}
		e(a,t)=
		\begin{cases}
			e\left(0,t-\e a\right)\exp(-\int_{-a}^{0}\zeta(a+s,t+\e s)ds) +\\
			\quad + \e\int_{-a}^{0}\dt\rhoz(a+s,t+\e s)\exp(-\int_{s}^{0}\zeta(a+\tau,t+\e \tau)d\tau)ds, &  a < t/\e, \\
			
			e\left(a-t/\e,0\right)\exp(-\int_{t/\e}^{0}\zeta(a+s,t+\e s)ds) +\\
			\quad +\e\int_{-t/\e}^{0}\dt\rhoz(a+s,t+\e s)\exp(-\int_{s}^{0}\zeta(a+\tau,t+\e \tau)d\tau)ds,&  a > t/\e. \\
		\end{cases}
	\end{equation}
	One has that
	\begin{equation}
		\begin{aligned}
			\int_{0}^{+\infty} (1+a)e(a,t)da&=\int_{0}^{t/\e} (1+a)e(a,t)da+\int_{t/\e}^{+\infty} (1+a)e(a,t)da\\
			&:=I_1+I_2.
		\end{aligned}
	\end{equation}
	We treat each term separately because they correspond to two cases {distinguished in} \eqref{duhamel.formula}~:
	\begin{equation*}
		\begin{aligned}
			I_1(t)&\leq \int_{0}^{\tse} (1+a)\left|e(0,t-\e a)\right|\exp\left( -\int_{-a}^{0}\zeta(a+s,t+\e s)ds\right)da\\
			&+\e \int_{0}^{\tse}(1+a)\int_{-a}^{0}\dt\rhoz(a+s,t+\e s)\exp\left( -\int_{s}^{0}\zeta(a+\tau,t+\e \tau)d\tau\right)dsda.
		\end{aligned}
	\end{equation*}
Using Theorem \ref{thm.error.estimates} and Corollary \ref{corom.moment.initial.layer}, 
one has that~:
	$$
	 \left| e(0,t-\e a)\right|\lesssim o_{\e}(1)+\int_{\tse -a} ^{\infty}m(a)da,
	 $$
	which thanks to Proposition \ref{prop.moments} gives that
	$$
	\begin{aligned}
		I_1(t)\leq& o_{\e}(1)\int_{0}^{\tse}(1+a)m(a)da+\int_{0}^{\tse}(1+a)m(a)\int_{\tse-a}^{\infty}m(\tia)d\tia da\\
		&+\e c_{0,1}\int_{\rr}(1+a)^3m(a)da, 
	\end{aligned}
	$$
	using similar argument as in the proof of Corollary \ref{corom.moment.initial.layer}, one shows that
	$$\int_{0}^{\tse}(1+a)m(a)da\int_{\tse-a}^{\infty}m(\tia)d \tia da\leq o_{\e}(1),$$
	since $(1+a)m(a)$ is integrable and by Lebesgue's Theorem $\int_{t}^{\infty}m(a)da \to 0$ as $t$ grows large. On the other hand,
$$
		\begin{aligned}
			I_2&\leq \int_{\tse}^{\infty}(1+a)\left| e(a-\tse,0)\right|\exp\left( -\int_{-\tse}^{0}\zeta(a+s,t+\e s)ds\right)da\\
			&+\e \int_{\tse}^{\infty}(1+a)\int_{-\tse}^{0}\left| \dt \rhoz(a+s,t+\e s)\right|\exp\left( -\int_{s}^{0}\zeta(a+\tau,t+\e \tau)d\tau\right)dsda\\
			&\leq c\int_{\tse}^{\infty}(1+a)m(a)da+\e \int_{\rr}(1+a)^3m(a)da.
		\end{aligned}
$$
\end{proof}
\subsection{Convergence results for the position}

\begin{thmm}\label{error.estimate.general}
	
	Under Assumptions \ref{hypo.data.general}, if $\rhoe$ is a solution of \eqref{eq.rho.complet.bis}, $\rhoz$ solves \eqref{eq.rho.zero}, $X_\e$ is a solution of \eqref{eq.X.eps} and $X_0$ solves \eqref{eq.outer.exp.zero} then
	$$ 
	\|X_\e-X_0\|_{C([0,T])}\leq o_\e(1).
	$$
\end{thmm}

\begin{proof}
		Setting $\veps(a,t) := (\Xeps(t)  - \Xeps(t-\e a))/\e$ where $\xeps(t)=\xp(t)$ when $t<0$, $\veps$ solves in a weak sense \cite{MiOel.2}~:
	\begin{equation}\label{eq.veps}
		\left\{
		\begin{aligned}
			& (\e \dt + \da) \veps = \dt \Xeps = \frac{1}{\muze} \left\{ \e \dt f + \int_\rr \veps(a,t) \zeta(a,t) \rhoe(a,t) da \right\},& a>0,\; t>0,		\\
			&	\veps(0,t) = 0, & a=0,\; t>0,	\\
			&			\veps(a,0) = \vepsi(a) := \frac{\Xeps(0)-\Xp(-\e a)}{\e}, & a>0,\; t=0,
		\end{aligned}
		\right.
	\end{equation}
since problem \eqref{eq.X.eps} can be expressed in a integro-differential equation :
$$
\mu_{0,\e} (t) \dt \Xeps = \e \dt f + \int_\rr \left(\frac{ \Xeps(t) - \Xeps(t-\e a)}{\e}\right)\zeta (a,t) \rhoe (a,t) da. 
$$
Following \cite[Theorem 6.1]{MiOel.2}, one has that
	$$
	\begin{aligned}
		\int_\rr  \left|\veps(a,t)\right| & \rhoe(a,t) da 
		& \leq  \int_0^t \left| \dt f(\tau)\right| d\tau + \left|f(0)\right| + L_{\Xp} \mu_{1,\max} =: c_1
	\end{aligned}
	$$
Moreover, 	
	one has also the bound :
	$$
	\| \dt \Xeps \|_{L^\infty (0,T)} \leq \frac{1}{\mu_{0,\min}} \left\{ \e  \| \dt f \|_{L^\infty (0,T)} + \ztmax c_1  \right\} =: c_2
	$$
	which provides thanks to Ascoli-Arzella that there exists a converging sub-sequence $\Xeps$ 
	in $C([0,T])$. Moreover, it $t>\e a$, using Duhamel's principle, 
	$$
	\left|\veps(a,t)\right| \leq \ue \int_{t- \e a}^t\left| \dt \Xeps(\tau) \right|d\tau \leq  c_2 a
	$$
	whereas if $t< \e a$, by similar arguments, 
	$$
	\begin{aligned}
		\left|\veps(a,t)\right| 
		& \leq \tse c_2 + \left| \frac{f(0)}{\muzmin}\right| + L_{\Xp} \frac{\mu_{1,\max}}{\muzmin}
	\end{aligned}
	$$
Thus, $\veps (a,t)/(1+a) \in L^\infty (\rr \times0,T)$ uniformly with respect to $\e$. These results
provide that $\veps$ weak-* converges in $L^\infty (\rr\times(0,T);(1+a)^{-1})$ 
to $\vz$ a weak solution of
\begin{equation}\label{eq.vz}
	\left\{
	\begin{aligned}
		& 	\da \vz = \dt \Xz = \frac{1}{\mu_{0,0}} \int_\rr \vz{(a,t)} \zeta(a,t) \rhoz(a,t) da 	\\ 
		&	\vz (0,t) = 0			
	\end{aligned}
	\right.
\end{equation}
which shows that 
$\vz(a,t) = a \Xz(t)$. Since \eqref{eq.X.eps} reads as 
$$
A_\e := \int_0^T \int_\rr \veps(a,t) \rhoe(a,t) da \varphi(t)dt = \int_0^T f(t) \varphi(t) dt,\quad \forall \varphi \in L^1(0,T),
$$
ans since $\dt \Xeps \wscvg \dt \Xz$ weak-$\star$ in $L^\infty(0,T)$ together with
$$
\rhoe \to \rhoz \text{ in } L^1(\rr\times(0,T);(1+a)) , \quad \veps \wscvg \vz \text{ in } L^\infty(\rr \times (0,T); (1+a)^{-1})
$$
one concludes that
$$
A_\e \to \int_0^T \int_\rr \vz(a,t) \rhoz(a,t) da \varphi(t) dt = \int_0^T \mu_{0,1} (t) \dt \Xz (t) \varphi dt = \int_0^T f(t) \varphi(t) dt .
$$
\end{proof}

\section{The linkages' density is a generic function constant in time}\label{section.3}

In this section, we consider problem \eqref{eq.X.eps} with a given kernel $\kernel$ constant in time. 
It is not supposed to solve any particular problem, what only matters are its moments with respect to $a$.
Define {the delay operator $\cL_\e$ as }
$$
\cL_\e[X](t) := \ue \left\{\mu_0 X(t) - \int_0^{\tse} X(t-\e a) \kernel(a) da\right\}.
$$
Then problem \eqref{eq.X.eps} can be rephrased as :
\begin{equation}\label{eq.X.operator.form}
	\cL_\e [\Xeps](t) = f(t)+ \ue \int_\tse^\infty \Xp(t-\e a ) \kernel(a)da, \quad \forall t>0	
\end{equation}
and we aim at constructing {an asymptotic expansion} $\ov{X}_{\e,N}$ such that it satisfies 
\begin{equation}\label{eq.tixen}
	\cL_\e[{\ov{X}}_{\e,N}] = f(t)+ \ue \int_\tse^\infty \Xp(t-\e a ) \kernel(a) da + O(\e^N).	
\end{equation}
{and we aim at showing that actually it is $\e^{N+1}$ accurate {\em i.e.}
	$\| \Xeps - \ov{X}_{\e,N}\|_{C([0,T])}\lesssim \e^{N+1}$.}

Here,  capital letters $(X_i)_{i\in \NN}$ denote the macroscopic correctors
defined on $[0,T]$, {they are independent on $\e$}, and the  microscopic correctors  $(x_{i,j})_{(i,j)\in \NN^2}$ or $(w_k)_{k\in \NN}$ are defined on {$\rr$}. They are then 
renamed with a tilde when rescaled with respect to~$\e$~: $\tix_{i,j}(t):=x_{i,j}(t/\e)$ for $t\in (0,T)$ and $(i,j)\in \N^2$. 
The  functions $\tix_{i,j}(t)$ correct tails associated to 
\begin{equation}\label{eq.def.Theta}
	{\Rjk(t) :=\left(\tse\right)^j \int_\tse^\infty a^k \kernel(a) da=:\rjk\left(\tse\right),}
\end{equation}
solving at the microscopic scale
$$
\cL_1[x_{i,j}](t):=\mu_0 x_{i,j} (t) - \int_0^t x_{i,j}(t-a) \kernel(a) da = \rjk\left(t\right)
$$
while $\tiw_{\ell}$ take into account other tails related to kernel $\kernel$ reading 
$$
\cL_1[w_\ell](t) = \mu_0 w_\ell (t) - \int_0^t w_\ell(t-a) \kernel(a) da =  \int_t^\infty  (t-a)^\ell \kernel(a) da.
$$
These tails occur when expanding the delay operator through Taylor series.
We give ourselves some hypotheses on the data~: 

\begin{assum}\label{hypo.data.cst}
	Assume that :
	\begin{compactenum}[i)]
		\item the source term is such that $f \in C^N(\rr)$.
		\item the past condition $X_p\in C^{N+1}(\rr).$	
		\item {we assume that $\kernel$ is a measurable non-negative fonction, {\em i.e.} $\kernel : \rr \to \rr$. Moreover, }
		for all $a \in \rr$, there exists $M \subset (a,\infty)$, $M$ compact and $|M|>0$
		such that $\kernel(\tia) >0$ for almost every $\tia \in M$.
		\item {defining the $k^{\text{th}}$ order moment of $\kernel$ as 
			$\mu_k := \int_\rr a^k \kernel(a) da$, for $k \in \NN$, one assumes that 
			$\mu_{k} < \infty$ for $k \in \{0,\dots,N+2\}$.}
		
	\end{compactenum}
\end{assum}

\subsection{Construction of the expansion}
{In order to approximate the solution $\Xeps$ of \eqref{eq.X.eps} within the framework defined using the previous hypotheses,}  we  {construct} of the terms forming the $N^{th}$-order approximation of $X_\e$ solution of \eqref{eq.X.eps} for a fixed kernel as
\begin{equation}\label{N.asymp.exp}
	{\ov{X}}_{\e,N} := \underbrace{\sum_{i=0}^{N-1} \e^i X_i(t)}_{\text{outer expansion}} + \underbrace{Y_N(t) + Z_N(t) + W_N(t)}_{\text{inner expansion}} ,
\end{equation}
where these terms are set later on.

\begin{propm}\label{prop.cle.xi}
	{Let Assumtions \ref{hypo.data.cst} hold}, let the sequence of functions $(X_i)_{i\in \{0,\dots,N-1\}}$ be given and for all $i\in \{0,\dots,N-1\}$ assume that 
	$X_i \in W^{N+1,\infty}$ $([0,T])$, then one has the expansion~: 
	$$
	\begin{aligned}
		\cL_\e[X_i](t) 
		= &\sum_{k=1}^{N-i}\frac{\e^{k-1} (-1)^{k+1} }{k!} X_i^{(k)}(t) \left(\mu_k - \Rk(t)\right) 
		+ \ue \int_\tse^\infty \kernel(a) da X_i(t)
		\\
		&+ \e^{N-i} \cR_i^{N+1-i}, 
	\end{aligned}
	$$
	and the rest can be controlled :
	$$
	\left|\cR_i^{N+1-i}\right| \leq \frac{1}{(N+1-i)!}\nrm{X^{(N+1-i)}_i}{\infty} 
	{	\mu_{N+1-i}} 
	$$
\end{propm}
\begin{proof}
	One writes :
	$$
	\cL_\e[X_i](t) 
	= \ue \int_0^\tse (X_i(t)-X_i(t-\e a))\kernel(a) da + \ue \int_\tse^\infty \kernel(a) da~ X_i(t) 
	$$
	then using the Taylor expansion~:
	$$
	\begin{aligned}
		X_i(t-\e a) = & \sum_{k=0}^{N-i} \frac{\e^{k} a^k }{k!} (-1)^k X_i^{(k)}(t) \\
		&+ \frac{(-\e a)^{N+1-i}}{(N-i)!}\int_0^1 X_i^{(N+1-i)}( t- s \e a 
		)  (1-s)^{N-i} ds
	\end{aligned}
	$$
	so that the first term above becomes~:
	$$
	\begin{aligned}
		\ue \int_0^\tse&  \left(X_i(t)-X_i(t-\e a)\right) \kernel(a) da \\
		& = \sum_{k=1}^{N-i} \int_0^\tse \frac{\e^{k-1} a^k }{k!} \kernel(a)da (-1)^{k+1} X_i^{(k)}(t) +\e^{N-i}\cR_i^{N+1-i} \\
		&= \sum_{k=1}^{N-i} \frac{(-1)^{k+1}}{k!} X_i^{(k)}(t) \left( \e^{k-1}\mu_k - \ue \int_\tse^\infty (\e a)^{k} \kernel(a) da \right)+\e^{N-i}\cR_i^{N+1-i} 
	\end{aligned}
	$$
	which provides the result.
\end{proof}
\begin{propm}[Outer expansions]\label{def.outer.exp}
	
	Under the same hypothesis as above, the zeroth-order macroscopic limit is given by
	\begin{equation}\label{eq.outer.exp.zero}
		\mu_1 X_0^{(1)} = f,
	\end{equation}
	and at any order $\ell \in \{1,\dots,N\}$, we have~:
	\begin{equation}\label{eq.outer.correctors}
		{\mu_1 X'_{\ell-1} =  \sum_{k=2}^\ell   (-1)^{k} \frac{\mu_{k} }{k!}X^{(k)}_{\ell-k}. 
		}
	\end{equation}
\end{propm}
\begin{proof}
	The result proved in Proposition \ref{prop.cle.xi} leads to ~:
	\begin{equation}\label{eq.restes}
		\begin{aligned}
			\cL_\e\left[\sum_{i=0}^{N-1}\e^i X_i \right] = \sum_{i=0}^{N-1} \e^i &
			\sum_{k=1}^{N-i}\frac{\e^{k-1} (-1)^{k+1} }{k!} X_i^{(k)}(t) \left(\mu_k - \Rk(t)\right) \\&
			+ \sum_{i=0}^{N-1} \e^{i-1}\int_\tse^\infty \kernel(a) da~ X_i(t)
			+ S_{N,0},
		\end{aligned}
	\end{equation}
	where we set $S_{N,0}:=\e^{N} \sum_{i=0}^{N-1}\cR_i^{N+1-i}$ and 
	$$
	\left|S_{N,0}\right| \leq 
	{\e^{N}}\max\limits_{i\in\{0,\dots,N-1\}} \left\{\mu_{N+1-i} \nrm{X_i}{W^{N+1-i,\infty}(0,T)} \right\}.$$
	Considering the first sum gives~:
	$$
	\begin{aligned}
		\sum_{i=0}^{N-1} \e^i &
		\sum_{k=1}^{N-i}\frac{\e^{k-1} (-1)^{k+1} }{k!} X_i^{(k)}(t) \left(\mu_k - \Rk(t)\right) \\
		& =
		\sum_{k=1}^N \sum_{\ell = k}^N \e^{\ell-1}\frac{ (-1)^{k+1} }{k!} X_{\ell-k}^{(k)}(t) \left(\mu_k - \Rk(t)\right)\\
		&		=\sum_{\ell=1}^N \e^{\ell-1} \sum_{k=1}^\ell \frac{ (-1)^{k+1} }{k!} X_{\ell-k}^{(k)}(t) \left(\mu_k - \Rk(t)\right)
	\end{aligned}
	$$
	Separating powers of $\e$ and considering that terms containing functions $\Rk$ belong
	to the initial layer (these depend only on the microscopic variable $t/\e$) provides~: 
	\begin{equation}\label{eq.outer.correctors.general}
		\sum_{k=1}^\ell  \mu_{k} \frac{X^{(k)}_{\ell-k}(t)}{k!} (-1)^{k+1}  = 
		\begin{cases}
			0 & \text{if}\;  \ell \neq 1, \\
			f{(t)} & \text{otherwise,}
		\end{cases}
	\end{equation}
	and by relating the lowest derivative with the highest index to the rest of the correctors, we establish 
	macroscopic nested ODEs  \eqref{eq.outer.exp.zero} and \eqref{eq.outer.correctors}.
\end{proof}
\begin{remark}
	The initial conditions of the macroscopic correctors $X_i$ are to be defined later (cf Theorem \ref{thm.ci.macro}). 
\end{remark}
\begin{propm}[Inner expansion]\label{inner.exp.def}
	It is threefold.
	\begin{itemize}
		\item The first part accounts for terms containing $\Rk$ in the first sum of \eqref{eq.restes}~:
		\begin{equation}\label{eq.def.YN}
			Y_N (t) :=  \sum_{m=1}^N \e^m \sum_{q=1}^m \sum_{k=1}^q \frac{(-1)^{k+1}}{k!(m-q)!} X^{(k+m-q)}_{q-k}(0) \tix_{m-q,k}(t) 
		\end{equation}
		where $\tix_{j,k}{(t)}:=x_{j,k}(t/\e)$ and the microscopic correctors solve~:
		\begin{equation}\label{eq.correcteur.micro}
			\cL_1 [ x_{j,k} ] (t) =\rjk(t) := t^j \int_t^\infty a^k \kernel(a) da, 
		\end{equation}
		and $\cL_1$ {is $\cL_\e$ defined for $\e=1$.}
		\item The second part corrects the second sum in \eqref{eq.restes} and reads~:
		\begin{equation}\label{eq.def.ZN}
			Z_N(t):=
			-\sum_{m=1}^{N}  \e^{m}\sum_{q=0}^{m-1}  \frac{1}{(m-q)!} X^{(m-q)}_{q}(0)  \tix_{m-q,0}{(t)} - \sum_{i=0}^{N-1} \e^{i} X_{i}(0){\tix_{0,0}(t) }.
		\end{equation}
		{and we underline that by uniqueness of the solution of \eqref{eq.correcteur.micro}, $\tix_{0,0}(t)=1$.}
		\item The last part concerns the remainders related to the past source term in \eqref{eq.tixen}
		$$
		W_N (t) := \sum_{i=0}^{N} \frac{\e^i}{i!} X_p^{(i)}(0) \tiw_{i}(t),
		$$
		where $\tiw_i(t) := w_i(t/\e)$
		and $(w_\ell)_\ell$ solve for $\ell \in \NN$,
		\begin{equation}\label{eq.past.correctors}
			\left\{
			\begin{aligned}
				&\int_\rr (w_\ell(t)-w_\ell(t-a) ) \kernel(a) da = 0, &  t>0, \\
				&w_\ell(t) = t^\ell ,& t\leq 0.
			\end{aligned}
			\right.
		\end{equation}
	\end{itemize}
\end{propm}

\begin{proof}
	First, we begin by constructing the first part of the initial layer $Y_N$. We consider the second term in \eqref{eq.restes} and we use  Taylor's expansion~:
	$$ X_{\ell-k}^{(k)}(t)=\sum_{j=0}^{N-k} 
	X_{\ell -k }^{(k+j)} (0) \frac{t^j}{j!}+ \e^N \cI^{N}_{k,l},
	$$
	where the integral rest reads :
	$$\cI^{N}_{k,l}:=\Frac{\e^{-k+1}}{(N-k)!}\left( \tse\right)^{N-k+1}\int_{0}^{1}(1-s)^{N-k}X_{\ell-k}^{(N-k+1)}(st) ds,$$
	which implies that
	\begin{equation}\label{eq.corectionYN}
		\begin{aligned}
			\sum_{\ell=1}^N &\e^{\ell-1} \sum_{k=1}^\ell \frac{ (-1)^{k} }{k!} X_{\ell-k}^{(k)}(t)  \Rk(t)  \\
			& =\sum_{\ell=1}^N \e^{\ell-1} \sum_{k=1}^\ell \frac{ (-1)^{k} }{k!} \Rk(t) \left\{ \sum_{j=0}^{N-k} 
			X_{\ell -k }^{(k+j)} (0) \frac{t^j}{j!}+ {\e^{N}} \cI^{N}_{k,l}
			\right\} \\
			& =\sum_{\ell=1}^N  \sum_{k=1}^\ell \sum_{j=0}^{N-k}  
			\frac{ (-1)^{k} }{k!j!}  \e^{\ell+j-1}
			X_{\ell -k }^{(k+j)} (0)\left(\tse\right)^j  \Rk(t) + S_{N,1} =: I + S_{N,1} 
		\end{aligned}
	\end{equation}
	where
	\begin{equation}\label{SN1}
		S_{N,1} := \sum_{\ell=1}^N  \sum_{k=1}^\ell \e^{\ell+N-1} \frac{(-1)^k}{k!} \Rk(t) \cI^{N}_{k,l},
	\end{equation}
	{Since on $(t/\e,\infty)$, $a>t/\e$ and since $N-k+1 \geq 0$ :
		\begin{equation}\label{eq.est.tails}
			\left(\tse\right)^j  \Rk(t)  \equiv \left(\frac{t}{\e}\right)^{N-k+1} \int_\tse^\infty a^k \kernel(a) da \leq \int_\tse^\infty a^{N+1} \kernel(a) da \leq \mu_{N+1}
		\end{equation}
		one has that :
		\begin{equation}\label{est.S.N.1}
			\left|S_{N,1}\right| \leq \e^N  C \mu_{N+1} \sum_{\ell=1}^N  \sum_{k=1}^\ell \e^{\ell-k} \leq 
			C' \e^N \mu_{N+1}
		\end{equation}
		where the generic constants $C$ and $C'$ depend only on 
		$\max_{j \in \{0,\dots,N-1\}} \nrm{X_{j}}{W^{1,N}(0,T)} .
		$
	}
	The first triple sum in \eqref{eq.corectionYN} can be decomposed thanks to Proposition \ref{prop.sum.integers} as
	$$
	\begin{aligned}
		I:= & \sum_{m=1}^N \e^{m-1} \left(\sum_{q=1}^{m} \sum_{k=1}^{q}  \frac{(-1)^{k}}{k!(m-q)!} X^{(k+m-q)}_{q-k}(0)\Ru{m-q,k}(t) \right)\\
		& + \sum_{m=N+1}^{2N-1} \e^{m-1}\left( \sum_{q=m+1-N}^N  
		\sum_{k=1}^{q+N-m} \frac{(-1)^{k}}{k!(m-q)!}X^{(k+m-q)}_{q-k}(0)\Ru{m-q,k}(t)\right) \\
		& =: I_1 + O(\e^N)
	\end{aligned}
	$$
	{where we recall the definition of $\Ru{j, k}$ in \eqref{eq.def.Theta}}.
	In order to compensate $I_1$,  
	we define microscopic correctors $\tix_{j,k}$ as \eqref{eq.correcteur.micro} and set $Y_N$ 
	as in \eqref{eq.def.YN}.
	Now, we need to correct the third term in \eqref{eq.restes}, which we do with the same technique as above~:
	$$
	\begin{aligned}
		&		\sum_{i=0}^{N-1} \e^{i-1} X_i(t)   \int_\tse^\infty \kernel(a) da\\
		&= 
		\sum_{i=0}^{N-1} \e^{i-1} \left\{ \sum_{j=0}^{N-i} \frac{t^j}{j!}X^{(j)}_{i}(0)  + \frac{ t^{N+1-i}}{(N-i)!}\int_0^1 X^{N+1-i}_i(st) (1-s)^{N-i} ds \right\} \int_\tse^\infty \kernel(a) da\\
		& = \sum_{i=0}^{N-1}  \sum_{j=0}^{N-i} \frac{\e^{i+j-1}}{j!} X^{(j)}_{i}(0) 
		\left(\tse\right)^j
		\Xi_{0,0}(t) + S_{N,2} \\
		& 		= \sum_{i=0}^{N-1}  \sum_{j=0}^{N-i}  \frac{\e^{i+j-1}}{j!} \Ru{j,0}(t)  X^{(j)}_{i}(0)  + S_{N,2} \\
		& = \sum_{m=1}^{N}  \e^{m-1}\sum_{q=0}^{m-1}  \frac{1}{(m-q)!} \Ru{m-q,0}(t)  X^{(m-q)}_{q}(0)   + \sum_{i=0}^{N-1} \e^{i-1}\Ru{0,0}(t)  X_{i}(0)  + S_{N,2} 
	\end{aligned}
	$$
	where $\Ru{j,0}(t) := \ru{j,0}(t/\e)$ and 
	\begin{equation}\label{def.S.N.2}
		S_{N,2}(t) :=  \sum_{i=0}^{N-1} \e^{i-1} \frac{t^{N+1-i}}{N-i!}\int_0^1 X^{N+1-i}_i(st) (1-s)^{N-i} ds \int_\tse^\infty \kernel(a)da, 
	\end{equation}
	and one has using the same argument as in \eqref{eq.est.tails}~:
	\begin{equation}\label{est.S.N.2}
		\left| S_{N,2}\right| \leq C \e^{N} \int_\rr (1+a)^{N+1} \kernel(a) da \sup_{i\in \{0,\dots,N\}}\nrm{X_i}{W^{N+1-i,\infty}(0,t)}.
	\end{equation}
	It suffices then to add the correction $Z_N$ defined as in \eqref{eq.def.ZN}.
	
	\noindent {Lastly, the right hand side of \eqref{eq.X.operator.form} contains
		a multi-scale term related to the past positions $\xp$. As it depends on 
		the fast variable $\tse$, we  Taylor-expand} $X_p(t- \e a)$ around $0$~:
	$$
	\begin{aligned}
		X_p(t-\e a)&  = \sum_{i=0}^{N} \e^i X_p^{(i)} (0) \frac{(t/\e- a)^i}{i!} \\
		&+ \e^{N+1} \frac{(t/\e- a)^{N+1}}{N!}\int_0^1 X_p^{(N+1) } (s(t-\e a)) (1-s)^{N} ds,
	\end{aligned}
	$$
	{leading to re-write the integral term in the rhs of \eqref{eq.X.operator.form} as}
	$$ \ue \int_{\tse}^{+\infty}X_p(t-\e a)\kernel(a)da=\sum_{i=0}^{N} \e^{i-1} X_p^{(i)} (0) \int_{\tse}^{+\infty}\frac{(t/\e- a)^i}{i!} \kernel(a)da+ S_{N,3},$$
	where 
	\begin{equation}\label{def.S.N.3}
		S_{N,3}:=\e^{N} \int_{\tse}^{+\infty}\frac{(t/\e- a)^{N+1}}{N!}\int_0^1 X_p^{(N+1) } (s(t-\e a)) (1-s)^{N} ds \kernel(a)da,
	\end{equation}
	and 
	\begin{equation}\label{err.est.SN3}
		\begin{aligned}
			\left|S_{N,3}\right|&\leq \frac{\e^{N}}{(N+1)!}\|X^{(N+1)}_p\|_{\infty}\int_{\tse}^{+\infty}\left( \tse -a\right)^{N+1}\kernel(a)da {\leq C \e^{N} \mu_{N+1}}
		\end{aligned}
	\end{equation}
	{This shows why the past microscopic correctors should reads as in the last part of the claim.}
\end{proof}


\subsection{Asymptotic behavior of initial layers}

\begin{lemm}\label{lem.propr.xjk}
	Assume that the integers $j$ and $k$ are chosen s.t. $j+k\leq N$, moreover, 
	if $\mu_{j+k+2}<\infty$, then one has~:
	$$
	x_{j,k}(0) = \begin{cases}
		\frac{\mu_k}{\mu_0},& \text{if} \; j=0, \\
		0,& \text{otherwise},
	\end{cases}
	$$
	and $x_{j,k} (t) \to \mu_{j+1+k}/((j+1)\mu_1)$ when $t \to \infty$.
\end{lemm}
{The proof of this lemma strongly relies on the concept of resolvent introduced and 
	deeply studied in \cite[Chap. 2, 4, 7]{Gripenberg_ea}}.
\begin{proof}
	The resolvent associated to \eqref{eq.correcteur.micro} \cite[Theorem 2.3.1 and 2.3.5, p. 42-45]{Gripenberg_ea}, satisfies~:
	$$
	\reso(t)-(\reso\star k)(t) = k(t),
	$$
	where 
	\begin{equation}\label{def.kernel.cst}
		k(a) := \kernel(a)/\mu_0, \quad (\reso \star k )(t) := \int_0^t \reso (t-\tau) k(\tau) d\tau 
	\end{equation}
	and it can be decomposed  \cite[Theorem 7.4.1, p.201]{Gripenberg_ea} as 
	$$
	\reso(t) = \frac{\mu_0}{\mu_1} + \gamma(t), 
	$$
	where the function $\gamma \in L^1(\rr)$. Moreover, the resolvent being defined the
	solution $x_{j,k}$ is {explicit} and reads~:
	$$
	x_{j,k} =\rjk +\rjk \star \reso =\rjk +\rjk \star (\mu_0/\mu_1 + \gamma) =\rjk +\rjk \star \gamma 
	+ \frac{\mu_0}{\mu_1}\int_0^t\rjk(s) ds.
	$$
	Thus the leading term in $x_{j,k}$ when $t$ grows large is the last integral. Indeed  
	$$
	\frac{\mu_0}{\mu_1} \int_\rr t^j \int_t^\infty a^k k(a)da dt = \frac{1}{\mu_1}\int_\rr \left(\int_0^a t^j dt\right) a^k \kernel(a) da = \frac{1}{(j+1)} \frac{\mu_{j+k+1}}{\mu_1}
	$$
	and one has {provided that $\mu_{j+k+2}<\infty$}~:
	$$
	x_{j,k} - \frac{1}{(j+1)} \frac{\mu_{j+k+1}}{\mu_1} \in L^1(\rr).
	$$
	which we define as the formal expression 
	$x_{j,k} (t) \to \mu_{j+1+k}/((j+1)\mu_1)$ when $t \to \infty$.
\end{proof}
\begin{lemm}\label{lem.propr.wl}
	Under the same assumptions as in the previous Lemma, the microscopic functions $w_\ell$ are discontinuous at $t=0$, for all $\ell\geq 0$~:
	$$
	w_\ell(0^+) = (-1)^\ell \frac{\mu_\ell }{\mu_0}, \quad w_\ell (0^-) = 0,
	$$
	and 
	and $w_{\ell} (t) \to ((-1)^\ell\mu_{\ell+1})/((\ell+1)\mu_1)$ when $t \to \infty$, {provided that
		$\mu_{\ell+2}<\infty$}.
\end{lemm}
\begin{proof}
	Using \eqref{eq.past.correctors}, we can easily show the discontinuity of the correctors $w_\ell$ at $t=0$.
	By the same arguments as proof of Lemma \ref{lem.propr.xjk}, one has that~:
	$$
	\begin{aligned}
		\lim_{t\to \infty} w_\ell (t) & = \frac{1}{\mu_1} \int_\rr \int_t^\infty (t-a)^{\ell} \kernel(a) da dt  = \frac{1}{\mu_1}\int_\rr \left( \int_0^a (t-a)^\ell dt\right)  \kernel(a) da \\
		& = \frac{(-1)^{\ell}}{(\ell+1)} \frac{\mu_{\ell+1}}{\mu_1},
	\end{aligned}
	$$
	which ends the proof.
\end{proof}
\begin{lemm}
	Under the previous {hypotheses, the initial data can be estimated}
	$$
	\left|X_\e (0^+) - \tiXen{N} (0^+)\right| \lesssim \e^{N{+1}}.
	$$
\end{lemm}
\begin{proof}
	By definition, one has~:
	\begin{equation}\label{eq.sum.CI}
		\tiXen{N} (0^+) = \sum_{i=0}^{N-1} \e^i X_i(0) + Y_N(0) + Z_N(0) + W_N(0^+).
	\end{equation}
	{Next we analyze each of these terms. Firstly, }one has thanks to 
	$$
	\begin{aligned}
		Y_N (0) = 
		\frac{1}{{\mu_0} }\sum_{\ell=1}^N  \e^{\ell} \sum_{k=1}^\ell 
		\frac{ (-1)^{k+1} }{k!} 
		X_{\ell -k }^{(k)} (0){\mu_k} 
		& = \e \frac{f(0)}{\mu_0} 
	\end{aligned}
	$$
	where we used that $\tix_{j,k}(0)=0$ for all $j\neq0$ and \eqref{eq.outer.correctors.general}.
	By definition, 
	$$
	Z_N(0) = - \sum_{i=0}^{N-1} \e^{i} X_i(0) \tix_{0,0}(0) = - \sum_{i=0}^{N-1} \e^{i} X_i(0),
	$$
	and compensates the first terms of the sum \eqref{eq.sum.CI}, then 
	\begin{equation}\label{eq.W.zeo}
		W_N(0) = \frac{1}{\mu_0}\sum_{i=0}^{N} \e^i \frac{X_p^{(i)}(0)}{i!} (-1)^i \mu_i. 	
	\end{equation}
	Then one observes that {if $\xeps$ solves \eqref{eq.X.operator.form}, then when $t\downarrow 0^+$, }
	$$
	X_\e(0^+) = \e \frac{f(0)}{\mu_0} + \int_\rr X_p(-\e a) \frac{\kernel(a) }{\mu_0}da 
	$$
	{and so an $N^{\text{th}}$-order Taylor expansion of {$\xp(-\e a)$ in 0} inside the last term  coupled 
		with \eqref{eq.W.zeo} leaves an $O(\e^{N+1})$ rest.}
\end{proof}
\subsection{Matching inner and outer expansions}
So far the initial conditions of the outer expansion are not defined. For this sake, we write the inner expansion's limit
when $t\to \infty$. This gives~:
$$
\lim_{t\to \infty} Y_N(t) = {\sum_{m=1}^N \e^{m} }
\sum_{q=1}^m \sum_{k=1}^{q}  \frac{(-1)^{k+1}}{(m-q+1)!k!} X^{(k+m-q)}_{q-k} (0)\frac{ \mu_{m-q+k+1} }{{\mu_1}}
=:s_1,
$$
together with~:
$$
\lim_{t\to \infty} Z_N(t) = 
-
\sum_{m=1}^{N} \e^{m} \sum_{q=0}^{m-1}  X^{(m-q)}_{q} \frac{\mu_{m-q+1}}{(m-q+1)!\mu_1} - \sum_{i=0}^{N-1}\e^{i} X_i(0) 
$$
and 
$$
\lim_{t\to \infty} W_N(t) = \sum_{i=0}^{N} \e^i\frac{X_p^{(i)}(0)}{(i+1)!} (-1)^i \frac{\mu_{i+1}}{\mu_1}=:s_3.
$$
As we do not want the inner expansion to interfere with the outer expansion, we
gather powers of $\e$ and define the initial conditions of the outer expansion so that
$$
\lim_{t\to \infty} \left(Y_N(t) + Z_N(t) + W_N(t)\right) = 0
$$
which then gives~:
\begin{thmm}\label{thm.ci.macro}
	The macroscopic Ansatz should be given the initial conditions~: 
	for $m=0$, $X_0(0) = X_p(0)$ while 
	for $m\in \{1,\dots,N-1\}$
	$$
	\begin{aligned}
		\mu_{1} X_{m}(0) &= (-1)^{m} X_p^{(m)}(0) \frac{\mu_{m+1}}{((m+1)!)} - \sum_{q=0}^{m-1} X^{(m-q)}_{q}(0) \frac{\mu_{m-q+1}}{(m-q+1)!} \\
		& + \sum_{q=1}^m \sum_{k=1}^{q}  \frac{(-1)^{k+1}}{(m-q+1)!k!} X^{(k+m-q)}_{q-k} (0) \mu_{m-q+k+1}.
	\end{aligned}
	$$
\end{thmm}
\subsection{Error estimates}
In this section, we give error estimates between $X_\e$, the  solution of \eqref{eq.X.operator.form} 
and the asymptotic expansion $\tiXen{N}$ given by \eqref{N.asymp.exp}. {This result is based 
	on a comparison principle \cite[Chap. 9, Section 8]{Gripenberg_ea} and the construction of 
	a super solution $U_N$ such that $|X_\e(t)-\tiXen{N}(t)|\leq U_N(t)$ for all $t \in [0,T]$. 
	Moreover, the super solution should satisfy $U_N\lesssim \e^N$ as well. 
	In order to use such a principle,  a strict control of the kernel 
	in some norm is expected. The following lemma is thus required~:}
\begin{lemm}\label{kernel}
	Under the Assumptions \ref{hypo.data.cst}, setting   
	$K_{\e}(t,s)$ $:=\frac{1}{\e \mu_{0}}$$ \kernel\left(\frac{s}{\e}\right)$ satisfies : 
	$$
	\vertiii{K_\e}_{B^{\infty}(0,T)}:=\esssup_{t \in (0,T)}\int_{0}^{t}\left|K_{\e}(t,s) \right|ds<1 .$$
\end{lemm}
{We refer the reader to \cite[Chap. 9, Sections 1 to 5]{Gripenberg_ea} for a complete overview of non-convolution integral Volterra equations
	and the corresponding functional framework.}
\begin{proof}For almost every $t\in (0,T)$,
	\begin{equation}\label{eq.int.K}
		0\leq \int_{0}^{t}\left|K_{\e}(\tilde{a},t) \right|~d\tilde{a}=
		\frac{\int_{0}^{\frac{t}{\e}}\kernel(a)~da}{\int_\rr\kernel(a)~da}\leq \frac{\int_{0}^{\frac{T}{\e}}\kernel(a)~da}{\int_\rr\kernel(a)~da}.
	\end{equation}
	But, by definition, for every fixed $\e$ there exists a compact set ${M_\e} \subset (T/\e,\infty)$ such that $\kernel(a) >0$ for 
	almost every $a \in M_\e$ so that
	$$
	\int_\rr \kernel(a)da - \int_{0}^{\frac{T}{\e}}\kernel(a)~da = \int_{\frac{T}{\e}}^\infty \kernel(a) da \geq \int_{{M_\e}} \kernel(a) da >0.
	$$
	which ends the proof.
\end{proof}

\begin{thmm}\label{thm.main.first.part}
	Suppose that Assumptions \ref{hypo.data.cst} holds, then~:
	$$
	\nrm{X_\e- {\ov{X}_{\e,N}}}{C([0,T])} \lesssim \e^N.
	$$
	where ${\ov{X}_{\e,N}}$ is defined in \eqref{N.asymp.exp} and $X_\e$ solves \eqref{eq.X.operator.form}.
\end{thmm}

\begin{proof}
	First, we consider the zero order approximation ({\em i.e. $N=1$}). We denote $\hX{1} := \Xeps - {\ov{X}_{\e,1}}$, it solves~:
	$$
	\cL_\e [\hX{1}] = S_1 := \sum_{i=1}^{3} S_{1,i},
	$$
	where $\left|S_1\right| \leq \e {C}_1$ (see  estimates \eqref{est.S.N.1}, \eqref{est.S.N.2} and \eqref{err.est.SN3} for $N=1$). We construct a super-solution $U_1$ such that 
	$$
	\cL_\e\left[\left|\hX{1}\right|\right]\leq \cL_\e[U_1] ,\quad \text{and}\; \left|\hX{1}(0)\right| \leq U_1(0).
	$$
	We set 
	$$
	U_1(t) := \e \left( c_1 + t c_2 -  \e c_3 \tiw_{1}(t)\right),
	$$
	where $\tiw_{1}$ solves \eqref{eq.past.correctors} with $\ell=1$.
	As the resolvent associated to \eqref{eq.past.correctors} is non-negative, 
	applying the comparison principle \cite[Propostion 9.8.1 and Lemma 9.8.2, p. 257]{Gripenberg_ea}, 
	shows that $\tiw_{1}(t) \leq 0$ for $t>0$. Then 
\begin{equation}\label{eq.comp}
\begin{aligned}
&		\cL_\e [U_1] =  c_1 \Rz(t) + \e c_2 \mu_1 + c_2 \int_\tse^\infty (t-\e a) \kernel(a) da -   \e c_3\cL_\e[\tiw_{1}] \\
& 		= c_1 \Rz(t) + \e c_2 \mu_1  + \e c_2 \int_\tse^\infty  \left(\tse - a\right) \kernel(a) da - \e c_3  \int_\tse^\infty\left(\tse - a\right) \kernel(a) da \geq \e c_2 \mu_1. 
\end{aligned} 	
\end{equation}
	The last inequality being true when $c_1 \geq 0$ and $c_2=c_3$. Then, one tunes $c_2 \geq  {C}_1/\mu_1$ so that
	$$
	\cL_\e\left[ \left| \hX{1} \right|\right] \leq \left|S_1\right| \leq \e K_1 \leq \e \mu_1 c_2 \leq \cL_\e[U_1],
	$$
	and the constant $c_1$ is chosen such that $\left|\hX{1}(0)\right| \leq \e c_1 \leq \e c_1 + \e^2 c_3 \frac{\mu_1}{\mu_0}= U_1(0)$. {The result then follows using \cite[Proposition 9.8.1 and Lemma 9.8.2]{Gripenberg_ea}.}

	More generally,  for any $N$, one sets $U_N := \e^N (c_1 + c_2 t - \e c_3 \tiw_{1})$ and the result follows the same~: choosing $c_1 := \left| \hX{N}(0) \right|$, $c_2 := K_N/\mu_1$ where $\left|S_N\right|\leq \e^N {C}_N$ and $c_3 =c_2$.
\end{proof}
{\begin{rmkm}\label{rmk.super.sol}
	We underline the novelty in the construction 
	of the super-solution $U$  when compared with \cite[Theorem 1.1]{MiOel.1}.
	Indeed, in this reference, terms
	$c_1 \Rz(t)$ and $ c_2 \int_\tse^\infty (t-\e a) \kernel(a) da$  in \eqref{eq.comp} 
	were compared so that the difference remain positive, and this 
	translates into the control of the ratio $A_\e [\rhoe] (t)$ 
	presented in the introduction.
	
	Here, instead, the negative tail is taken into account 
	by the initial layer $-\e c_3 \tiw_1$ and there is no need 
	to relate the previous terms anymore. This in turn allows
	to relax hypotheses on the decay of $\kernel$.
\end{rmkm}}

{\section{Conclusion \& perspectives}
	
	In this work we have obtain two-fold results. On one hand, we have weakened hypotheses
	on the death rate of the linkages and showed that still when the parameter $\e$ is small
	adhesive memory effects are close to friction. 
	
	Moreover, we made the previous statement more quantitative with respect to 
	the size of the parameter $\e$ thanks to a general asymptotic expansion
	when $\kernel$ is constant in time.
	
	To our knowledge, it is not possible yet to combine both parts of this article and 
	construct  asymtotic expansions (and error estimates) at any order $N\geq 2$, for the weakly coupled problem $(\rhoe, \Xeps)$
	solving \eqref{eq.rho.complet.bis} and \eqref{eq.X.eps}. The main reason for this
is that at some point we use arguments based on Laplace transform of the resolvent (indirectly 
this leads to the asymptotic decomposition of the resolvent used in Lemmas \ref{lem.propr.xjk} and \ref{lem.propr.wl} thanks to \cite[Theorem 7.4.1]{Gripenberg_ea}). When the density is time 
dependent, the microscopic correctors $x_{j,k}(s,t)$ an $w_\ell(s,t)$ 
depend on a macroscopic parameter $t$, ($s$ bien the local microscopic variable)
(this is usual in the framework of quasi-periodic asymptotic expansions \cite{CioDo})
and so do the corresponding resolvents $r_{x_{j,k}}(s,t)$ and $r_{w_\ell}(s,t)$.  Results provided by \cite[Theorem 7.4.1 p. 201]{Gripenberg_ea}
do not allow for uniform estimates of the initial layers with respect to the macroscopic 
parameter $t$ and the analysis performed in Section \ref{section.3} can't be extended to 
time dependent kernels. 
This seems an interesting mathematical problem  {\em per se}, since for the time 
being, asymptotic results established thanks to the Laplace transform, are
to our knowledge,  the sharpest possible under the weakest hypotheses. 
}

\appendix

\section{Some summations over integers}

\begin{figure}[ht!]
	\centering
	\begin{tikzpicture}[thick,scale=0.95, every node/.style={scale=0.95}]
		\draw [-stealth,very thick](-1,0) -- (5.25,0);
		\node at (5.5,0) (j) {$j$};
		\draw [-stealth,very thick](0,-1) -- (0,5.25);
		\node at (0,5.5) (i) {$i$};
		\node[anchor=east] at (0,4) (N) {$N$} ;
		\node[anchor=south west,rotate = -45,color=red] at (0,0) (mj1) {$m=1$};
		\node[anchor=south west,rotate = -45,color=red] at (2,2) (mjN) {$m=N$};		
		\draw (-1pt,4) -- (1pt,4) ;
		\node[anchor=south east] at (0,0) (i1) {1} ;
		\draw (-1pt,0) -- (1pt,0) ;
		\node[anchor=north] at (4,0) (nm1) {$N-1$} ;
		\node[anchor=north west] at (0,0) (z) {0};
		\draw [-stealth,color=red] (-0.5,-0.5) -- (5,5) ;
		\node[
		color=red] at (5.75,4.75) (m) {$m=i+j$};
		\draw [-stealth,color=blue] (4.5,-0.5) -- (-0.5,4.5) ;
		\node[
		color=blue] at (-1.25,4.75) (n) {$n=i-j+1$};
		\draw[-stealth,color=green] (-0.5,4)--(5.5,4);
		\node[anchor=west,color=green
		] at (5.5,4) (iegalN) {$n=2N+1-m$};
		\draw[-stealth,color=green] (4,-0.5)--(4,5.5);
		\node[anchor=north,
		color=green,
		] at (4,5.85) (jegalNm1) {$n=m-2N+3$};
		\foreach \x in {0,1,2,3,4,5}
		\foreach \y in {0,1,2,3,4,5}
		{\draw (\x-0.05,\y) -- (\x+0.05,\y);
			\draw (\x,\y-0.05) -- (\x,\y+0.05);
		}
	\end{tikzpicture}
	\caption{The index change from $(i,j)$ to $(m,n)$ (here as an example $N=5$). {It allows to separate the $\e$-scales as powers of $m$. We look for a peculiar way of   going through integer values of $(i,j)$  (crossed dots in the figure) although 
			this does not correspond to all integer values of $(m,n)$.
	 }}
	\label{fig:indexchange}
\end{figure}

\begin{propm}\label{prop.sum.integers}
	Assume that $N\geq 2$, and define the sum $S$ as follows 
	$$
	S := \sum_{i=1}^N \sum_{k=1}^i \sum_{j=0}^{N-k} a_{i,j,k}
	$$
	where $a_{i,j,k}$ is a sequence of real numbers,
	then this sum is in fact equal to
	\begin{equation}\label{eq.S.claim}
			S = \left(\sum_{m=1}^N \sum_{q=1}^{m} \sum_{k=1}^{q}  + \sum_{m=N+1}^{2N-1} \sum_{q=m+1-N}^N
		\sum_{k=1}^{q+N-m} \right)a_{q,m-q,k} 
	\end{equation}
\end{propm}

\begin{proof}
	{We distinguish between  $i<N$ and $i=N$. }
	In the first case, one has, by the following computation,
		that
		$$
		s:= \sum_{i=1}^{N-1} \sum_{k=1}^i \sum_{j=0}^{N-k} a_{i,j,k}  =  \sum_{i=1}^{N-1} \sum_{j=0}^{N-i-1} \sum_{k=1}^i  a_{i,j,k}+ 
		\sum_{i=1}^{N-1} \sum_{j=N-i}^{N-1} \sum_{k=1}^{N-j} a_{i,j,k},
		$$
		because
		$$
		\left\{
		\begin{aligned}
			&	0\leq j \leq N-i-1, \quad \min(i,N-j) = i \\
			& N-i \leq j \leq N-1, \quad \min(i,N-j) = N-j
		\end{aligned}
		\right.
		$$
		One has that 
		$$
		s{ =  \sum_{i=1}^{N-1} \left(\sum_{j=0}^{N-i-1} +
		\sum_{j=N-i}^{N-1} \right) \sum_{k=1}^{\min(i,N-j)} a_{i,j,k} }
		=\sum_{i=1}^{N-1} \sum_{j=0}^{N-1} \sum_{k=1}^{\min(i,N-j)} a_{i,j,k}.
		$$ 
		On the other hand, when $i=N$,  a simple check shows that
		$$
		s':=\sum_{k=1}^N \sum_{j=0}^{N-k} a_{N,j,k} = \sum_{j=0}^{N-1} \sum_{k=1}^{N-j} a_{N,j,k},
		$$
		and then one remarks simply that $\min(i=N,N-j)=N-j$ which gathering the terms gives 
		that 
		$$
		S = s+s' = \sum_{i=1}^N \sum_{j=0}^{N-1} \sum_{k=1}^{\min(i,N-j)} a_{i,j,k}.
		$$

		The previous sum can be rewritten as~:
		\begin{equation}\label{eq.S}
			\begin{aligned}
				S = & \left(\sum_{m=1}^{N} \sum_{n=3-m}^{m+1} 
				+ \sum_{m=N+1}^{2N-1} \sum_{n=m-2N+3}^{2N-m+1}\right)  \\
				& \sum_{k=1}^{\min(i(m,n),N-j(m,n))}
				a_{i(m,n),j(m,n),k} \chiu{\{i(m,n),j(m,n)\in\N^2\}}(m,n),
			\end{aligned}
		\end{equation}
		where  $m:=i+j$ and $n:=i-j+1$ 
		and the inverse transform should provide integer values
		$i(m,n):=(m+n-1)/2$ and $j(m,n):=(m-n+1)/2$ (see Fig. \ref{fig:indexchange}).
	{	This justifies the definition of the indicator function 
		$$
		\chiu{\{i(m,n),j(m,n)\in\N^2\}}(m,n) := \begin{cases}
			1, & \text{ if } (i(m,n),j(m,n))\in\N^2, \\
			0, & \text{ otherwise. }
		\end{cases}
		$$
		}
		When $m>N$, one needs to bound the summation
		on $n$ in an interval depending on $m$ (see Fig. \ref{fig:indexchange}). Indeed, 
		when $i=N$, we write~:
		$$
		n=N-j+1,\quad m=N+j,\quad \Rightarrow n=2N -m +1,
		$$
		while if $j=N-1$, 
		$$
		n=i-N+2,\quad m=i+N-1,\quad \Rightarrow n=m-2N +3,
		$$
		a simple check shows that $n\leq 2N-1 \Leftrightarrow m-2N+3 \leq 2N-m+1$.
		
		Then since the {indicator function $\chiu{\{(i(m,n),j(m,n))\in \NN^2\}}$ in \eqref{eq.S} is not zero} when $[n+m-1]_2=0$, there
		exists $q \in \Z$ such that 
		$n+m-1 = 2 q$ 
		or equivalently $n=1+2q-m$
		so that the summation with respect to $n$ can be {replaced} with a summation over $q$.
		When 
		$n \in \{3-m,m+1\}$, $q \in \{1,m\}$, and similarly when $n \in \{m-2N+3,2N-m+1\}$, 
		$q \in \{m+1-N,N\}$. Moreover 
		$q=i$ and $j=m-q$.
		Thus the previous sum becomes~:
		$$
		S =\left(\sum_{m=1}^N   \sum_{q=1}^m + \sum_{m=N+1}^{2N-1}\sum_{q=m+1-N}^N \right)\sum_{k=1}^{\min(q,N-(m-q))} a_{q,m-q,k}
		$$
		Since $\min(q,q+N-m) = q + \min(0,N-m) = q$ as soon as $m\leq N$ {this gives the first part of the sum in \eqref{eq.S.claim}}.
		If $m \in \{N+1,\dots,2N-1\}$, then $N-m \leq  -1$ and $\min(q,q+N-m) = q+N-m$ which ends the proof.
\end{proof}

\begin{propm}\label{prop.sum.integers.Z}
	In the same way as above
	$$
	S':= \sum_{i=0}^{N-1} \sum_{j=1}^{N-i} a_{i,j} = \sum_{m=1}^N \sum_{q=0}^{m-1} a_{q,m-q}
	$$
\end{propm}
\begin{proof}
	Again we perform the change of variables $m=i+j$ and $q=i$ and we proceed as above.
\end{proof}

\section{Proof of Proposition \ref{prop.moments}} \label{sec.proof}
\begin{proof}
	First, for $j=0$, $\rhoz$ is explicitly given by 
	$$
	\rhoz(a,t) = \rhoz(0,t) \exp\left( - \int_0^a \zeta(\tia,t)d\tia\right) \leq \rhoz(0,t) \frac{m(a)}{m(0)} \leq \frac{\bmax \ztmax }{\ztmax+\bmin} \frac{m(a)}{m(0)}
	$$
	which gives $c_0$.
	Similarly, $ \dt \rhoz(a,t)$ it is explicit and reads~:
	$$
	\begin{aligned}
		\dt \rhoz(a,t) = \dt & \rhoz(0,t)\exp\left(-\int_0^a \zeta(\tia,t) d\tia\right) \\
		& - \int_0^a \exp\left( - \int_\tau^a \zeta(\tia,t) d \tia \right) \dt \zeta(\tau,t) \rhoz(\tau,t)d\tau, 
	\end{aligned}
	$$
	where
	$$ \dt \rhoz(0,t)=\frac{g(t)}{1+\beta(t)\int_0^{+\infty} \exp\left( - \int_\tau^a \zeta(\tia,t) d \tia \right)da}$$
	and 
	\begin{equation*}
		\begin{aligned}
			g(t)&= \beta'(t)\left( 1-\mu_0(t)\right)\int_0^{+\infty} \exp\left( - \int_\tau^a \zeta(\tia,t) d \tia \right)da\\
			&-\int_0^{+\infty}\int_0^a \exp\left( - \int_\tau^a \zeta(\tia,t) d \tia \right) \dt \zeta(\tau,t) \rhoz(\tau,t)d\tau da.
		\end{aligned}
	\end{equation*}
	So that
	$$
	\begin{aligned}
		\left|	\dt \rhoz(a,t)  \right| & \leq\left|\dt \rhoz(0,t)\right| \frac{m(a)}{m(0)} + m(a) \nrm{\zeta}{W^{1,\infty}}
		\int_0^a \frac{\rhoz(\tau,t)}{m(\tau)} d\tau \leq (k_1 + k_2 a )m(a) \\
		& \leq c'(1+a) m(a)
	\end{aligned}
	$$
	where 
	$$ k_1:=C\left( \Frac{\zeta_{\max}}{\beta_{\max}+\zeta_{\max}}, \|\beta\|_{W^{1,\infty}},\|m\|_{L^1(\rr)},\|\zeta\|_{W^{1,\infty}} \right), \hspace{0.2cm} k_2=c_0\Frac{\|\zeta\|_{W^{1,\infty}}}{m(0)}$$ 	
	and $ c_{0,1}:=\max(k_1,k_2)$.
	Now, for $j=1$, $\rhou$ can be given explicitly by
	$$
	\rhou (a,t) = \rhou (0,t) \exp\left(-\int_0^a \zeta(\tia,t) d\tia\right) - \int_0^a \exp\left( - \int_\tau^a \zeta(\tia,t) d \tia \right) \dt \rhoz(\tau,t)d\tau ,
	$$
	where
	$$ \rhou(0,t)=\frac{h(t)}{1+\beta(t)\int_0^{+\infty} \exp\left( - \int_\tau^a \zeta(\tia,t) d \tia \right)da}$$
	such that
	\begin{equation*}
		\begin{aligned}
			\left| h(t)\right|&=\left| \int_{\rr}\int_0^a \exp\left( - \int_\tau^a \zeta(\tia,t) d \tia \right) \dt \rhoz(\tau,t)d\tau da \right|\\
			&\leq \int_{\rr}\int_0^a \Frac{m(a)}{m(\tau)} \left| \dt \rhoz(\tau,t)\right|d\tau da\leq c_{0,1}\int_{\rr} (1+a)^2 m(a)da,
		\end{aligned}
	\end{equation*}
	and finally, we obtain that
	\begin{equation*}
		\begin{aligned}
			\left| \rhou(a,t)\right|&\leq k'_1 \frac{m(a)}{m(0)}+\int_{0}^{a}c_{0,1}(1+\tau)m(a)d\tau
			\leq \max(k'_1,c_{0,1})(1+a)^2m(a)
		\end{aligned}
	\end{equation*}
	where
	$$ k'_1:=C\left( \Frac{\zeta_{\max}}{\beta_{\max}+\zeta_{\max}}, \int_{\rr} (1+a)^2 m(a)da \right).$$
	Similarly, we can prove that
	$$ \left| \dt\rhou\right|\leq c_{1,1}(1+a)^3m(a),$$
	and the generic way can be deduced by induction.
\end{proof}
\section*{Acknowledgments}
This research was partly funded by Cofund Math in Paris, co-funded by the European Commission under the Horizon2020 research and innovation programme, Marie Sklodowska-Curie grant agreement No 101034255.

\section*{Conflict of interest}

The authors declare there is no conflict of interest.

\bibliographystyle{abbrv}
\bibliography{biblioVM}

\end{document}